\documentclass[12pt,a4paper]{article}
\usepackage{fullpage}
\usepackage[utf8]{inputenc}
\usepackage{amsfonts}
\usepackage{amsmath}
\usepackage{amssymb}
\usepackage{amsthm}
\usepackage{hyperref}
\usepackage[usenames]{color}
\usepackage{tikz}

\usepackage[normalem]{ulem}

\usetikzlibrary{intersections,through}
\usetikzlibrary{shapes.geometric, backgrounds, calc,
  arrows, automata, fit, positioning, patterns,
  decorations.pathreplacing, tikzmark, arrows.meta}
\tikzstyle{dot}=[fill, circle, minimum size=1.0ex, inner sep=0pt]
\tikzstyle{ml}=[line width=0.7pt, -]
\tikzstyle{mb}=[line width=2pt, -]
\tikzstyle{cb}=[blue]
\tikzstyle{cv}=[violet]
\tikzstyle{fc}=[fill=white!75!black]

\newtheorem{thm}{Theorem}[section]
\newtheorem{lem}[thm]{Lemma}
\newtheorem{dfn}[thm]{Definition}
\newtheorem{cor}[thm]{Corollary}
\newtheorem{rem}[thm]{Remark}
\newtheorem{ex}[thm]{Example}

\def\ip{\mathfrak{i}}

\def\so{\mathfrak{s}}
\def\no{\mathfrak{n}}
\def\P{P}
\def\IP{I}
\def\SO{S}
\def\SP{M}
\def\NO{N}

\newcommand\oeis[1]{\href{https://oeis.org/#1}{#1}}

\usepackage[affil-it]{authblk}
\author{Sergey \textsc{Kirgizov}}
\author{Khaydar \textsc{Nurligareev}}
\affil{LIB, Universit\'e de Bourgogne, France}
\date{}

\title{Asymptotics of self-overlapping permutations}

\begin{document}

\maketitle

\begin{abstract}
 In this work, we study the concept of self-overlapping permutations,
 which is related to the larger study of consecutive patterns in permutations.
 We show that this concept admits a simple and clear geometrical meaning, and prove that a permutation can be represented as a sequence of non-self-overlapping ones.
 The above structural decomposition allows us to obtain equations for the corresponding generating functions,
 as well as the complete asymptotic expansions for the probability that a large random permutation is (non-)self-overlapping.
 In particular, we show that almost all permutations are non-self-overlapping, and that the corresponding asymptotic expansion has the self-reference property:
 the involved coefficients count non-self-overlapping permutations once again.
 We also establish complete asymptotic expansions of the distributions of very tight non-self-overlapping patterns,
 and discuss the similarities of the non-self-overlapping permutations to other permutation building blocks, such as indecomposable and simple permutations, as well as their associated asymptotics.
\end{abstract}

{\bf Keywords\/}: {asymptotics, permutations, consecutive patterns, overlapping.}

\section{Introduction}

The study of \emph{pattern-avoiding permutations} is a dynamically developing area of mathematics that has been exciting scientists since the mid-1980s.
The central concept of this study is containment:
a~permutation $\sigma = \sigma_1 \sigma_2 \ldots \sigma_n$ \emph{contains} another permutation $\pi = \pi_1 \pi_2 \ldots \pi_k$ as a \emph{pattern}
if there exists a subsequence $1 \leqslant i_1 < i_2 < \ldots < i_k \leqslant n$ such that for all indices $j$ and $l$ the inequalities $\pi_j < \pi_l$ and $\sigma_{i_j} < \sigma_{i_l}$ are equivalent.
If $\sigma$ does not contain $\pi$, then it is said that $\sigma$ \emph{avoids} $\pi$.
Since the first systematic enumeration by Simion and Schmidt~\cite{SimionSchmidt1985}, the interest in permutation patterns has grown steadily, which can be explained by the presence of extensive connections with other mathematical fields.
For instance, the theorem of Marcus and Tardos \cite{MarcusTardos2004} stating that the growth rate of every permutation class avoiding a given pattern is singly exponential relates to words and 0--1 matrices as well.

One of the important variations of the above study concerns the notion of \emph{consecutive patterns}, first systematically treated by Elizalde and Noy~\cite{ElizaldeNoy2003}.
In terms of B\'ona~\cite{Bona2007}, we say that
a permutation $\sigma = \sigma_1 \sigma_2 \ldots \sigma_n$ \emph{tightly contains} a permutation $\pi = \pi_1 \pi_2 \ldots \pi_k$
if there exists an integer $0 \leqslant i \leqslant n-k$ such that for all indices $j$ and $l$ the inequalities $\pi_j < \pi_l$ and $\sigma_{i+j} < \sigma_{i+l}$ are equivalent.
A good overview of the state of this topic for 2016 was provided by Elizalde~\cite{Elizalde2016}.
To obtain generating functions and distributions of tight pattern occurrences in various combinatorial objects,
one can employ the \emph{cluster method} of Goulden and Jackson~\cite{GouldenJackson1979, GouldenJackson2004, NoonanZeilberger1999}.
As a~particular example, we would like to cite the 
proofs of Elizalde--Noy and Nakamura conjectures~\cite{Elizalde2013}.

Consecutive patterns can be thought as a restricted version of permutation patterns, where the entries of $\sigma$ that form a pattern must be in consecutive positions.
It is natural to take one step further by forcing these entries to be consecutive themselves.
Following B\'ona~\cite{Bona2007}, we say that
a permutation $\sigma = \sigma_1 \sigma_2 \ldots \sigma_n$ \emph{very tightly contains} another permutation $\pi = \pi_1 \pi_2 \ldots \pi_k$
if there exist two integers $0 \leqslant i,h \leqslant n-k$ such that $\sigma_{i+j} = \pi_j + h$ for all $1 \leqslant j \leqslant k$.
For the first time, these and more general patterns were considered by Myers~\cite{Myers2002} who called the corresponding type of containment \emph{rigid}.
Myers established distributions of non-self-overlapping patterns (see Definition~\ref{def:self-overlapping}). With the help of the cluster method, her result was completed by Claesson~\cite{Claesson2022} who also provided joint distributions for any patterns of the same size.
Note that Claesson called the patterns under consideration \emph{Hertzsprung}.
In the following, we will use the term \emph{very tight pattern} to refer to them.

Understanding the overlaps of combinatorial sub-structures is of core importance in the cluster method.
In the case of consecutive patterns, this need is reflected in the concept of non-overlapping permutations~\cite{Bona2011}, also known as minimally overlapping~\cite{DuaneRemmel2011,PanRemmel2016}:
a permutation of length $k$ is \emph{non-overlapping} if it is not contained at the same time as a prefix and suffix in any permutation of size $n$ satisfying $k < n < 2k-1$.
In other words, if the intersecting prefix and suffix of a permutation are isomorphic, then they represent an example of overlapping permutation.
Here, we call two sequences of distinct integers $\pi = \pi_1 \pi_2 \ldots \pi_k$ and $\tau = \tau_1 
\tau_2 \ldots \tau_k$ \emph{isomorphic}
if $\pi_i > \pi_j \Leftrightarrow \tau_i > \tau_j$ for all possible indices $i$ and~$j$.
For instance, the permutation~$132$ is non-overlapping, but
$1324$ is overlapping, since the prefix and suffix of size four of the permutation $132546$ form a~$1324$-pattern.
The counting sequence of non-overlapping permutations starts with
 \[
  1,\,
  2,\,
  4,\,
  12,\,
  48,\,
  280,\,
  1\,864,\,
  14\,840,\,
  132\,276,\,
  1\,323\,504,\,
  \ldots,
 \]
see \oeis{A263867} entry in Sloane's Encyclopedia~\cite{oeis}.

Applying the notion of very tight containment to overlapping permutations, we come to the main object of our interest.
This is about the concept of \emph{self-overlapping permutations} (called \emph{extendible} by Myers~\cite{Myers2002} and B\'ona~\cite{Bona2007}),
which is given by the following definition.

\begin{dfn}\label{def:self-overlapping}
 A permutation $\sigma \in S_n$ is \emph{self-overlapping}, if there is an integer $1 \leqslant k<n$ called \emph{overlapping range} such that the following three conditions hold:
 \begin{enumerate}
  \item
   the interval $\{1,\ldots,k\}$ is invariant under $\sigma$,
  \item
   the interval $\{n-k+1,\ldots,n\}$ is invariant under $\sigma$,
  \item
   the first and last $k$ consecutive positions of $\sigma$ form isomorphic patterns.
 \end{enumerate}
 Otherwise, we call $\sigma$ \emph{non-self-overlapping}.
 We designate by $\so_n$ and $\no_n$, respectively, the numbers of self-overlapping and non-self-overlapping  permutations of size $n$.
\end{dfn}

Here, we use the definition given by B\'ona~\cite{Bona2007}.
It has a particular property: if $\sigma = \sigma_1 \sigma_2 \ldots \sigma_n$ is self-overlapping, then $\sigma_1 < \sigma_n$.
Sometimes this property is too restrictive, since the reverse $\overline{\sigma} = \sigma_n \sigma_{n-1} \ldots \sigma_1$ behaves similarly to $\sigma$ in many senses.
This is why both Myers~\cite{Myers2002} and Claesson~\cite{Claesson2022} understood $\overline{\sigma}$ as a self-overlapping permutation as well.

From a geometric point of view, if a permutation $\sigma \in S_n$ is self-overlapping with an overlapping range $k$, then its plot possesses two congruent blocks of size $k \times k$:
one of them is located at the lower left corner, while the other is in the upper right one.
For example, the permutation $\sigma = 214365$ is self-overlapping with overlapping ranges $k=2$ and $k=4$ (Figure~\ref{fig:example_214365}).

\begin{figure}[ht!]
\centering
\begin{tikzpicture}
 \begin{scope}[xshift = -5cm, scale=0.5]
  \draw[ml] (0,0) grid (6,6);
  \begin{scope}
   \node[dot] (a) at (1-0.5,2-0.5) {};
   \node[dot] (b) at (2-0.5,1-0.5) {};
   \node[dot] (c) at (3-0.5,4-0.5) {};
   \node[dot] (d) at (4-0.5,3-0.5) {};
   \node[dot] (e) at (5-0.5,6-0.5) {};
   \node[dot] (f) at (6-0.5,5-0.5) {};
   \draw (1-0.5,-0.5) node {$2$};
   \draw (2-0.5,-0.5) node {$1$};
   \draw (3-0.5,-0.5) node {$4$};
   \draw (4-0.5,-0.5) node {$3$};
   \draw (5-0.5,-0.5) node {$6$};
   \draw (6-0.5,-0.5) node {$5$};
  \end{scope}
  \draw[mb,cb] (4,6) rectangle (6,4);
  \draw[mb,cv] (2,0) rectangle (0,2);
 \end{scope}
 \begin{scope}[scale=0.5]
  \draw[ml] (0,0) grid (6,6);
  \begin{scope}
   \node[dot] (a) at (1-0.5,2-0.5) {};
   \node[dot] (b) at (2-0.5,1-0.5) {};
   \node[dot] (c) at (3-0.5,4-0.5) {};
   \node[dot] (d) at (4-0.5,3-0.5) {};
   \node[dot] (e) at (5-0.5,6-0.5) {};
   \node[dot] (f) at (6-0.5,5-0.5) {};
   \draw (1-0.5,-0.5) node {$2$};
   \draw (2-0.5,-0.5) node {$1$};
   \draw (3-0.5,-0.5) node {$4$};
   \draw (4-0.5,-0.5) node {$3$};
   \draw (5-0.5,-0.5) node {$6$};
   \draw (6-0.5,-0.5) node {$5$};
  \end{scope}
  \draw[mb,cb] (2,6) rectangle (6,2);
  \draw[mb,cv] (4,0) rectangle (0,4);
 \end{scope}
\end{tikzpicture}
\caption{Overlapping blocks in the permutation $214365$.}
\label{fig:example_214365}
\end{figure}
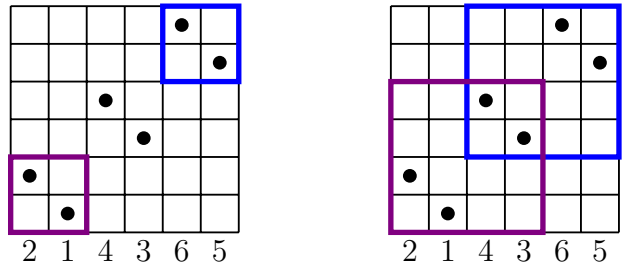

Figure~\ref{fig:small_self-overlapping_permutations} illustrates all self-overlapping permutation of size $n \leqslant 4$.
Observing these and other examples, the reader may give a guess that a self-overlapping permutation possesses an overlapping range of no more than half its size.
In Section~\ref{sec:structure}, we show that this is, indeed, the case (Lemma~\ref{lem:overlapping_range}).
This observation gives rise to a structural decomposition of permutations into direct sums of non-self-overlapping ones (Theorem~\ref{thm:full_decomposition}).
In its turn, the decomposition leads to relations on generating functions (Theorem~\ref{thm:ogf_for_NO_and_SO}), that allows us to establish counting sequences $(\so_n)$ and $(\no_n)$. 

\begin{figure}[ht!]
\centering
\begin{tikzpicture}
 \begin{scope}[xshift = -5.5cm, yshift = 0.5cm, scale=0.5]
  \draw[ml] (0,0) grid (2,2);
  \begin{scope}
   \node[dot] (a) at (1-0.5,1-0.5) {};
   \node[dot] (b) at (2-0.5,2-0.5) {};
   \draw (1-0.5,-0.5) node {$1$};
   \draw (2-0.5,-0.5) node {$2$};
  \end{scope}
 \end{scope}
 \begin{scope}[xshift = -3cm, yshift = 0.25cm, scale=0.5]
  \draw[ml] (0,0) grid (3,3);
  \begin{scope}
   \node[dot] (a) at (1-0.5,1-0.5) {};
   \node[dot] (b) at (2-0.5,2-0.5) {};
   \node[dot] (c) at (3-0.5,3-0.5) {};
   \draw (1-0.5,-0.5) node {$1$};
   \draw (2-0.5,-0.5) node {$2$};
   \draw (3-0.5,-0.5) node {$3$};
  \end{scope}
 \end{scope}
 \begin{scope}[xshift = 0cm, scale=0.5]
  \draw[ml] (0,0) grid (4,4);
  \begin{scope}
   \node[dot] (a) at (1-0.5,1-0.5) {};
   \node[dot] (b) at (2-0.5,2-0.5) {};
   \node[dot] (c) at (3-0.5,3-0.5) {};
   \node[dot] (d) at (4-0.5,4-0.5) {};
   \draw (1-0.5,-0.5) node {$1$};
   \draw (2-0.5,-0.5) node {$2$};
   \draw (3-0.5,-0.5) node {$3$};
   \draw (4-0.5,-0.5) node {$4$};
  \end{scope}
 \end{scope}
 \begin{scope}[xshift = 3.5cm, scale=0.5]
  \draw[ml] (0,0) grid (4,4);
  \begin{scope}
   \node[dot] (a) at (1-0.5,1-0.5) {};
   \node[dot] (b) at (2-0.5,3-0.5) {};
   \node[dot] (c) at (3-0.5,2-0.5) {};
   \node[dot] (d) at (4-0.5,4-0.5) {};
   \draw (1-0.5,-0.5) node {$1$};
   \draw (2-0.5,-0.5) node {$3$};
   \draw (3-0.5,-0.5) node {$2$};
   \draw (4-0.5,-0.5) node {$4$};
  \end{scope}
 \end{scope}
 \begin{scope}[xshift = 7cm, scale=0.5]
  \draw[ml] (0,0) grid (4,4);
  \begin{scope}
   \node[dot] (a) at (1-0.5,2-0.5) {};
   \node[dot] (b) at (2-0.5,1-0.5) {};
   \node[dot] (c) at (3-0.5,4-0.5) {};
   \node[dot] (d) at (4-0.5,3-0.5) {};
   \draw (1-0.5,-0.5) node {$2$};
   \draw (2-0.5,-0.5) node {$1$};
   \draw (3-0.5,-0.5) node {$4$};
   \draw (4-0.5,-0.5) node {$3$};
  \end{scope}
 \end{scope}
\end{tikzpicture}
\caption{Self-overlapping permutations of size at most $4$.}
\label{fig:small_self-overlapping_permutations}
\end{figure}
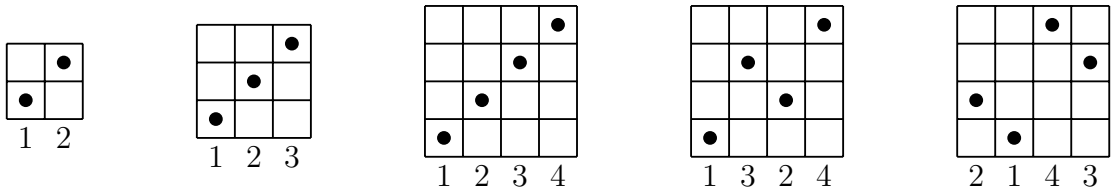

 The first few values of the sequence $(\so_n)$ are
 \[
  \so_n
   =
  0,\,
  1,\,
  1,\,
  3,\,
  7,\,
  31,\,
  131,\,
  775,\,
  5\,211,\,
  41\,315,\,
  \ldots
 \]
 while the complementary sequence $(\no_n)$ begins with
 \[
  \no_n
   =
  1,\,
  1,\,
  5,\,
  21,\,
  113,\,
  689,\,
  4\,909,\,
  39\,545,\,
  357\,669,\,
  3\,587\,485,\,
  \ldots
 \]
A single glance at these numbers is enough to conjecture that with high probability a typical large permutation is non-self-overlapping.
We prove this conjecture at Section~\ref{sec:asymptotics}.
Moreover, we establish a complete asymptotic expansion of the probability that a uniform random permutation is self-overlapping (Theorem~\ref{thm:asymp_so}).
It turns out that the coefficients involved into the expansion admit combinatorial interpretation: they are $\no_n$ again.
Remarkably, this result does not need the advanced techniques typically used to establish combinatorial meaning of asymptotic expansions~\cite{MonteilNurligareev2024,Nurligareev2022}.
The proof follows directly from the above mentioned structural decomposition, supplemented by estimates on the asymptotic behavior of the tails.

In Section~\ref{sec:pattern_asymptotics}, we discuss very tight patterns.
As we have already mentioned, for the non-self-overlapping case, the distributions of such patterns were established by Myers~\cite{Myers2002}.
Here, based on her results, we present complete asymptotic expansions of these distributions.
For the self-overlapping case, the distributions were given by Claesson~\cite{Claesson2022}.
Evaluating asymptotics requires advanced methods in this case, such as Borinsky's approach~\cite{Borinsky2018},
which is beyond the scope of this work.
We will provide the full result in our next paper~\cite{KN-clusters}.

The established asymptotics raise several questions that we discuss in Section~\ref{sec:conclusion}.

\section{Structure and enumeration}\label{sec:structure}

\begin{lem}\label{lem:overlapping_range}
 Any self-overlapping permutation of size $n$ admits an overlapping range of size at most $n/2$.
\end{lem}
\begin{proof}
  The idea of the proof can be illustrated by Figure~\ref{fig:example_214365}.
  As seen from the right-hand side of the figure, the permutation 214365 admits the overlapping range of size $4$,
  the corresponding overlapping blocks are 2143 and 4365.
  The intersection of these blocks is the block 43 of size $2$,
  which is isomorphic to the lower left corner of 2143, as well as to the upper right corner of 4365.
  Hence, we have two isomorphic blocks, 21 and 65, and the initial permutation admits the overlapping range of size $2$. 

 In the general case, suppose that an overlapping range $k$ of a self-overlapping permutation $\sigma \in S_n$ is greater than~$n/2$.
 In this case, the lower left and upper right blocks of size $k \times k$ are congruent and have non-empty intersection.
 The intersection is a $(2k-n)\times(2k-n)$ block, which is congruent to the lower left and upper right blocks of the same size.
 Therefore, $\sigma$ admits overlapping range $(2k-n)<k$.
 If $(2k-n)<n/2$, we are done.
 Otherwise, repeat the procedure until the overlapping block sizes are less than $n/2$ (Figure~\ref{fig:small_overlapping_range_exists}).
\end{proof}

\begin{figure}[ht!]
\centering
\begin{tikzpicture}
 \begin{scope}[xshift = -5cm, scale=0.5]
  \draw[ml] (0,0) rectangle (7,7);
  \draw[mb,fc] (1,6) rectangle (6,1);
  \filldraw[violet!25!white] (5,0) rectangle (0,5);
  \filldraw[blue!25!white] (2,7) rectangle (7,2);
  \filldraw[violet!25!black!40!white] (5,1) rectangle (1,5);
  \filldraw[blue!25!black!40!white] (6,2) rectangle (2,6);
  \filldraw[black!55!white] (5,2) rectangle (2,5);
  \draw[mb,cb] (1,7) rectangle (7,1);
  \draw[mb,cv] (6,0) rectangle (0,6);
  \draw[dashed,mb,cb] (2,7) rectangle (7,2);
  \draw[dashed,mb,cv] (5,0) rectangle (0,5);
 \end{scope}
 \begin{scope}[scale=0.5]
  \draw[ml] (0,0) rectangle (7,7);
  \draw[mb,fc] (2,5) rectangle (5,2);
  \filldraw[violet!25!white] (3,0) rectangle (0,3);
  \filldraw[blue!25!white] (4,7) rectangle (7,4);
  \filldraw[violet!25!black!40!white] (3,2) rectangle (2,3);
  \filldraw[blue!25!black!40!white] (5,4) rectangle (4,5);
  \draw[mb,cb] (2,7) rectangle (7,2);
  \draw[mb,cv] (5,0) rectangle (0,5);
  \draw[dashed,mb,cb] (4,7) rectangle (7,4);
  \draw[dashed,mb,cv] (3,0) rectangle (0,3);
 \end{scope}
 \begin{scope}[xshift = 5cm, scale=0.5]
  \draw[ml] (0,0) rectangle (7,7);
  \draw[mb] (4,3) rectangle (3,4);
  \draw[mb,cb] (4,7) rectangle (7,4);
  \draw[mb,cv] (3,0) rectangle (0,3);
 \end{scope}
 \draw (-0.75,1.75) node {$\rightsquigarrow$};
 \draw (4.25,1.75) node {$\rightsquigarrow$};
\end{tikzpicture}
\caption{Schema for reducing an overlapping range. If overlapping blocks have a~non-empty intersection (gray areas on the figure), then this intersection can be slided to get smaller overlapping blocks (blue and violet areas).}
\label{fig:small_overlapping_range_exists}
\end{figure}
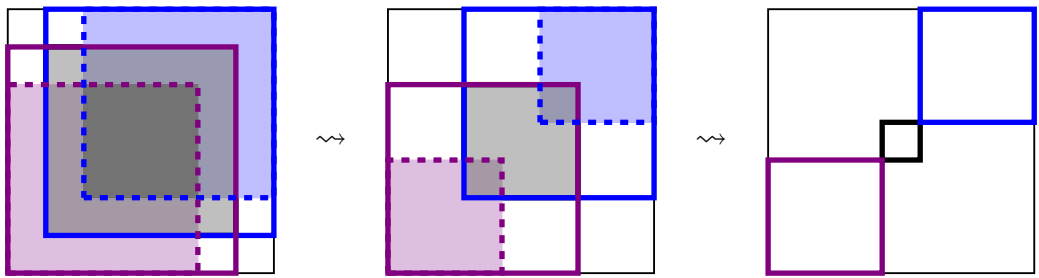

Recall that the \emph{direct sum} of permutations $\pi \in S_m$ and $\tau \in S_n$ is the permutation $\pi\oplus\tau$ of length $(m+n)$ defined by
 \[
  (\pi\oplus\tau)(i) =
  \left\{\begin{array}{ll}
   \pi(i) & \mbox{for } 1\leqslant i \leqslant m, \\
   \tau(i-m)+m & \mbox{for } m+1\leqslant i \leqslant n.
  \end{array}\right.
 \]
Lemma~\ref{lem:overlapping_range} allows us to decompose self-overlapping permutations into a direct sum of smaller permutations
(schematically, the summands are represented by consecutive non-intersecting blocks, see the right part of Figure~\ref{fig:small_overlapping_range_exists};
decompositions of several permutations of size $6$ are shown in Figure~\ref{fig:several_decompositions}).
More precisely, we get the following structural result.

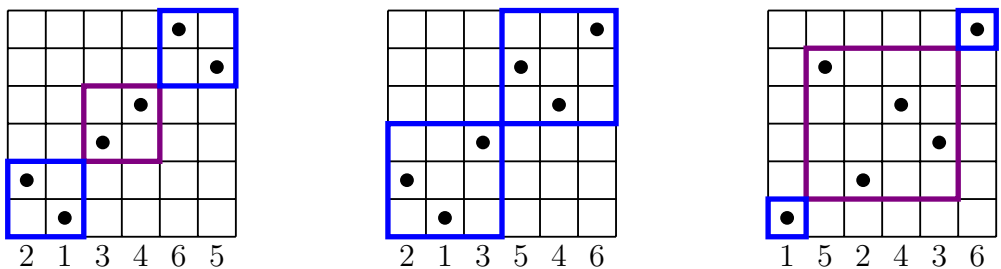
\begin{figure}[ht!]
\centering
\begin{tikzpicture}
 \begin{scope}[xshift = -5cm, scale=0.5]
  \draw[ml] (0,0) grid (6,6);
  \begin{scope}
   \node[dot] (a) at (1-0.5,2-0.5) {};
   \node[dot] (b) at (2-0.5,1-0.5) {};
   \node[dot] (c) at (3-0.5,3-0.5) {};
   \node[dot] (d) at (4-0.5,4-0.5) {};
   \node[dot] (e) at (5-0.5,6-0.5) {};
   \node[dot] (f) at (6-0.5,5-0.5) {};
   \draw (1-0.5,-0.5) node {$2$};
   \draw (2-0.5,-0.5) node {$1$};
   \draw (3-0.5,-0.5) node {$3$};
   \draw (4-0.5,-0.5) node {$4$};
   \draw (5-0.5,-0.5) node {$6$};
   \draw (6-0.5,-0.5) node {$5$};
  \end{scope}
  \draw[mb,cv] (2,2) rectangle ++(2,2);
  \draw[mb,cb] (0,0) rectangle ++(2,2);
  \draw[mb,cb] (4,4) rectangle ++(2,2);
 \end{scope}
 \begin{scope}[scale=0.5]
  \draw[ml] (0,0) grid (6,6);
  \begin{scope}
   \node[dot] (a) at (1-0.5,2-0.5) {};
   \node[dot] (b) at (2-0.5,1-0.5) {};
   \node[dot] (c) at (3-0.5,3-0.5) {};
   \node[dot] (d) at (4-0.5,5-0.5) {};
   \node[dot] (e) at (5-0.5,4-0.5) {};
   \node[dot] (f) at (6-0.5,6-0.5) {};
   \draw (1-0.5,-0.5) node {$2$};
   \draw (2-0.5,-0.5) node {$1$};
   \draw (3-0.5,-0.5) node {$3$};
   \draw (4-0.5,-0.5) node {$5$};
   \draw (5-0.5,-0.5) node {$4$};
   \draw (6-0.5,-0.5) node {$6$};
  \end{scope}
  \draw[mb,cb] (0,0) rectangle ++(3,3);
  \draw[mb,cb] (3,3) rectangle ++(3,3);
 \end{scope}
 \begin{scope}[xshift = 5cm, scale=0.5]
  \draw[ml] (0,0) grid (6,6);
  \begin{scope}
   \node[dot] (a) at (1-0.5,1-0.5) {};
   \node[dot] (b) at (2-0.5,5-0.5) {};
   \node[dot] (c) at (3-0.5,2-0.5) {};
   \node[dot] (d) at (4-0.5,4-0.5) {};
   \node[dot] (e) at (5-0.5,3-0.5) {};
   \node[dot] (f) at (6-0.5,6-0.5) {};
   \draw (1-0.5,-0.5) node {$1$};
   \draw (2-0.5,-0.5) node {$5$};
   \draw (3-0.5,-0.5) node {$2$};
   \draw (4-0.5,-0.5) node {$4$};
   \draw (5-0.5,-0.5) node {$3$};
   \draw (6-0.5,-0.5) node {$6$};
  \end{scope}
  \draw[mb,cv] (1,1) rectangle ++(4,4);
  \draw[mb,cb] (0,0) rectangle ++(1,1);
  \draw[mb,cb] (5,5) rectangle ++(1,1);
 \end{scope}
\end{tikzpicture}
\caption{Decompositions of several self-overlapping permutations.}
\label{fig:several_decompositions}
\end{figure}

\begin{lem}\label{lem:SO-decomposition}
 Any self-overlapping permutation $\sigma$ can be uniquely decomposed into a direct sum
 \[
  \sigma
   =
  \pi \oplus \tau \oplus \pi,
 \]
 where $\pi$ is non-self-overlapping permutation and $\tau$ is arbitrary (possibly, empty) permutation.
\end{lem}
\begin{proof}
 Let $k$ be the minimal overlapping range of $\sigma$.
 According to Lemma~\ref{lem:overlapping_range}, $k \leqslant n/2$.
 Hence, $\sigma$~possesses three consecutive invariant intervals: $\{1,\ldots,k\}$, $\{k+1,\ldots,n-k\}$ and $\{n-k+1,\ldots,n\}$ (it may happen that the second of them is empty).
 This gives us the above decomposition $\sigma = \pi \oplus \tau \oplus \pi$.
 Note that the permutation $\pi \in S_k$ is non-self-overlapping due to the minimality of the overlapping range $k$.
 Indeed, if $\pi$ is self-overlapping with an overlapping range $l$, then $l<k$ is the overlapping range of $\sigma$ too, which leads to a contradiction (see Figure~\ref{fig:overlapping_range_minimality}).
\end{proof}

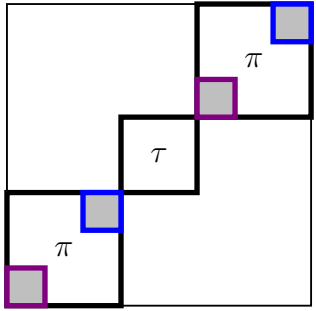
\begin{figure}[ht!]
\centering
\begin{tikzpicture}
 \begin{scope}[scale=0.5]
  \draw[ml] (0,0) rectangle (8,8);
  \draw[mb] (5,3) rectangle (3,5);
  \draw[mb] (5,8) rectangle (8,5);
  \draw[mb] (3,0) rectangle (0,3);
  \draw[mb,cv,fc] (1,0) rectangle (0,1);
  \draw[mb,cv,fc] (6,5) rectangle (5,6);
  \draw[mb,cb,fc] (3,2) rectangle (2,3);
  \draw[mb,cb,fc] (8,7) rectangle (7,8);
  \draw (2-0.5,2-0.5) node {$\pi$};
  \draw (4,4) node {$\tau$};
  \draw (7-0.5,7-0.5) node {$\pi$};
 \end{scope}
\end{tikzpicture}
\caption{Structure of self-overlapping permutations. If the overlapping block $\pi$ is a self-overlapping permutation itself, then its size is not the minimal overlapping range of the permutation $\sigma = \pi \oplus \tau \oplus \pi$.}
\label{fig:overlapping_range_minimality}
\end{figure}

\begin{cor}\label{cor:SO_finite}
 The counting sequence $(\so_n)$ of self-overlapping permutations satisfies 
 \[
  \so_n
   =
  \sum\limits_{k=1}^{\lfloor n/2\rfloor} \no_k \cdot (n-2k)!.
 \]
\end{cor}

 The decomposition in Lemma~\ref{lem:SO-decomposition} can be refined if the same lemma is applied to the middle term.
 Thus, we obtain a decomposition into a direct sum of non-self-overlapping permutations.
 For example, the permutation 124356879 can be represented as $1\oplus132\oplus1\oplus132\oplus1$. (Figure~\ref{fig:example_124356879}).
 More generally, the following theorem takes place.

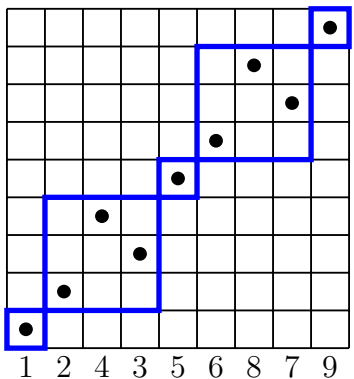
\begin{figure}[ht!]
\centering
\begin{tikzpicture}
 \begin{scope}[xshift = -5cm, scale=0.5]
  \draw[ml] (0,0) grid (9,9);
  \begin{scope}
   \node[dot] (a) at (1-0.5,1-0.5) {};
   \node[dot] (b) at (2-0.5,2-0.5) {};
   \node[dot] (c) at (3-0.5,4-0.5) {};
   \node[dot] (d) at (4-0.5,3-0.5) {};
   \node[dot] (e) at (5-0.5,5-0.5) {};
   \node[dot] (f) at (6-0.5,6-0.5) {};
   \node[dot] (g) at (7-0.5,8-0.5) {};
   \node[dot] (h) at (8-0.5,7-0.5) {};
   \node[dot] (i) at (9-0.5,9-0.5) {};
   \draw (1-0.5,-0.5) node {$1$};
   \draw (2-0.5,-0.5) node {$2$};
   \draw (3-0.5,-0.5) node {$4$};
   \draw (4-0.5,-0.5) node {$3$};
   \draw (5-0.5,-0.5) node {$5$};
   \draw (6-0.5,-0.5) node {$6$};
   \draw (7-0.5,-0.5) node {$8$};
   \draw (8-0.5,-0.5) node {$7$};
   \draw (9-0.5,-0.5) node {$9$};
  \end{scope}
  \draw[mb,cb] (1,1) rectangle ++(3,3);
  \draw[mb,cb] (5,5) rectangle ++(3,3);
  \draw[mb,cb] (0,0) rectangle ++(1,1);
  \draw[mb,cb] (4,4) rectangle ++(1,1);
  \draw[mb,cb] (8,8) rectangle ++(1,1);
 \end{scope}
\end{tikzpicture}
\caption{Decomposition of the permutation $124356879$.}
\label{fig:example_124356879}
\end{figure}

\begin{thm}\label{thm:full_decomposition}
 Any permutation $\sigma$ can be uniquely decomposed into a direct sum of non-self-overlap\-ping permutations,
 \begin{equation}\label{eqn:full_decomposition}
  \sigma
   =
  \pi_1
   \oplus \ldots \oplus
  \pi_m \oplus \tau \oplus \pi_m
   \oplus \ldots \oplus 
  \pi_1,
 \end{equation}
 where $m\geqslant0$ the permutation $\tau$ is, possibly, empty.
\end{thm}
\begin{proof}
 If $\sigma$ is non-self-overlapping, then decomposition~\eqref{eqn:full_decomposition} is done with $m=0$ and $\tau=\sigma$.
 Otherwise, let us apply Lemma~\ref{lem:SO-decomposition} iteratively.
 In other words, express $\sigma$ as $\pi_1 \oplus \tau_1 \oplus \pi_1$.
 If $\tau_1$ is self-overlapping, then express it as $\pi_2 \oplus \tau_2 \oplus \pi_2$, etc.
 Thus, after a finite number of iterations, we obtain decomposition~\eqref{eqn:full_decomposition}.

 Now, let us show the uniqueness of the above decomposition.
 Suppose the contrary, that is, suppose that we have two different decompositions of form~\eqref{eqn:full_decomposition}:
 \[
  \pi_1
   \oplus \ldots \oplus
  \pi_m \oplus \tau \oplus \pi_m
   \oplus \ldots \oplus 
  \pi_1
   =
  \pi'_1
   \oplus \ldots \oplus
  \pi'_{m'} \oplus \tau' \oplus \pi'_{m'}
   \oplus \ldots \oplus 
  \pi'_1.
 \]
 Without loss of generality, we can assume that $\pi_1 \neq \pi'_1$, which means that they are of different sizes, say, $l$ and $l'$.
 However, if $l<l'$, then $\pi'_1$ is self-overlapping with the overlapping range $l$ (see Figure~\ref{fig:decomposition_uniqueness}).
 The case $l'<l$ leads to a similar contradiction.
 Thus, $\pi_1=\pi'_1$, and decomposition~\eqref{eqn:full_decomposition} is unique.
\end{proof}

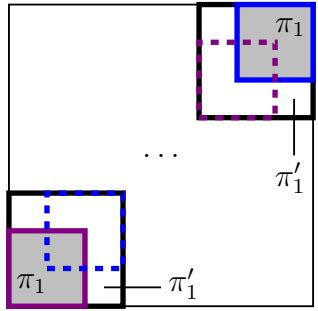
\begin{figure}[ht!]
\centering
\begin{tikzpicture}
 \begin{scope}[scale=0.5]
  \draw[ml] (0,0) rectangle (8,8);
  \draw[mb] (5,8) rectangle (8,5);
  \draw[mb] (3,0) rectangle (0,3);
  \draw[mb,cv,fc] (2,0) rectangle (0,2);
  \draw[mb,cb,fc] (8,6) rectangle (6,8);
  \draw[mb,cv,dashed] (7,5) rectangle (5,7);
  \draw[mb,cb,dashed] (3,1) rectangle (1,3);
  \draw (1-0.4,1-0.4) node {$\pi_1$};
  \draw (5-0.4,1-0.4) node {$\pi'_1$};
  \draw[ml] (2.5,0.5) -- ++(1.5,0);
  \draw (4,4) node {$\ldots$};
  \draw (8-0.6,8-0.6) node {$\pi_1$};
  \draw (8-0.6,4-0.6) node {$\pi'_1$};
  \draw[ml] (7.5,5.5) -- ++(0,-1.5);
 \end{scope}
\end{tikzpicture}
\caption{Uniqueness of decomposition~\eqref{eqn:full_decomposition}.}
\label{fig:decomposition_uniqueness}
\end{figure}

Let us consider the generating functions $\SO(z)$ and $\NO(z)$ of self-overlapping and non-self-overlap\-ping permutations, respectively:
 \[
  \SO(z) = \sum\limits_{n=1}^{\infty} \so_n z^n
  \qquad\mbox{and}\qquad
  \NO(z) = \sum\limits_{n=1}^{\infty} \no_n z^n.
 \]
We formally assume that an empty permutation is neither self-overlapping nor non-self-overlapping.
Hence, the generating function $\P(z)$ of all permutations can be represented as the following sum:
 \begin{equation}\label{eqn:P=1+SO+NO}
  \P(z) = 1 + \SO(z) + \NO(z) = \sum\limits_{n=0}^{\infty} n!\, z^n.
 \end{equation}
This observation, together with Lemma~\ref{lem:SO-decomposition}, helps us establish relations that determine the behavior of $\SO(z)$ and $\NO(z)$.

\begin{thm}\label{thm:ogf_for_NO_and_SO}
 The generating functions $\SO(z)$ and $\NO(z)$ satisfy the following relations:
 \begin{equation}\label{eqn:ogf_for_NO}
  \NO(z) = \P(z) \big(1 - \NO(z^2)\big) - 1
 \end{equation}
 and
 \begin{equation}\label{eqn:ogf_for_SO}
  \SO(z) = \dfrac{1 + \NO(z)}{1 - \NO(z^2)} \cdot \NO(z^2).
 \end{equation}
\end{thm}
\begin{proof}
 In terms of generating functions, Lemma~\ref{lem:SO-decomposition} can be interpreted as the following relation:
 \[
  \SO(z) = \P(z) \cdot \NO(z^2).
 \]
 At the same time, it follows from \eqref{eqn:P=1+SO+NO} that
 \[
  \SO(z) = \P(z) - \NO(z) - 1.
 \]
 Equating the right-hand sides of these two identities, we obtain relation~\eqref{eqn:ogf_for_NO}.
 Meanwhile, to get~\eqref{eqn:ogf_for_SO}, it is sufficient to take the first of the identities and replace $\P(z)$ with the expression obtained via relation~\eqref{eqn:ogf_for_NO}.
\end{proof}

\begin{rem}\label{rem:meaning_of_equations}
 Relation~\eqref{eqn:ogf_for_NO} rewritten as
 \begin{equation}\label{eqn:ogf_for_NO_revisited}
  \P(z) = \dfrac{1 + \NO(z)}{1 - \NO(z^2)}
 \end{equation}
 is another face of Theorem~\ref{thm:full_decomposition}.
 Indeed, the factor $1 + \NO(z)$ corresponds to the central block of decomposition~\eqref{eqn:full_decomposition},
 while $\big(1 - \NO(z^2)\big)^{-1}$ represents the sequence of non-self-overlapping permutations.
 Note that the argument $z^2$ reflects the fact that each element of the sequence is taken twice.
 Relation~\eqref{eqn:ogf_for_SO} carries a similar meaning.
 The only difference is the presence of the additional factor $\NO(z^2)$ that guarantees that the decomposition consists of at least two blocks.
\end{rem}

\section{Asymptotics}\label{sec:asymptotics}

\begin{thm}\label{thm:asymp_so}
 For any positive integer $r$, the probability that a uniform random permutation $\sigma \in S_n$ is self-overlapping, as $n\to\infty$, satisfies
 \begin{equation}\label{eqn:asymp_so}
  \mathbb{P}(\sigma\mbox{ is self-overlapping})
   = 
  \sum\limits_{k=1}^{r-1} \dfrac{\no_k}{n^{\underline{2k}}}
   +
  O\left(\dfrac{1}{n^{2r}}\right),
 \end{equation}
 where $n^{\underline{k}} = n(n-1)\ldots(n-k+1)$ are the falling factorials.
\end{thm}
\begin{proof}
 It follows from Corollary~\ref{cor:SO_finite}, that the counting sequence of self-overlapping permutations satisfies 
 \[
  \so_n
   \leqslant
  \sum\limits_{k=1}^{r-1} \no_k \cdot (n-2k)!
   +
  \sum\limits_{k=r}^{\lfloor n/2\rfloor} k! \cdot (n-2k)!.
 \]
 Let us estimate the second summand.
 For $n$ large enough and $r < k \leqslant n/2$, we have 
 \[
  k! \cdot (n-2k)! \leqslant (r+1)! \cdot (n-2r-2)!.
 \]
 As a consequence,
 \[
  \sum\limits_{k=r}^{\lfloor n/2\rfloor} k! \cdot (n-2k)!
   \leqslant
   r! \cdot (n-2r)! + (\lfloor n/2\rfloor-r) \cdot (r+1)! \cdot (n-2r-2)!
   =
  O\big((n-2r)!\big),
 \]
 which leads us to
 \[
  \so_n
   =
  \sum\limits_{k=1}^{r-1} \no_k \cdot (n-2k)!
   +
  O\big((n-2r)!\big).
 \]
 Now, since $\mathbb{P}(\sigma\mbox{ is self-overlapping}) = \so_n/n!$, it is sufficient to divide the obtained inequality by~$n!$. 
\end{proof}

 Let $\so_n^{(m)}$ be the number of permutations of size $n$ whose decomposition~\eqref{eqn:full_decomposition} consists of $(2m+1)$ summands
 (in particular, $\so_n^{(0)} = \no_n$).
 Define also $\no_n^{(m)}$ to be the number of permutations that can be represented as a direct sum of $m$ non-self-overlapping permutations (in particular, $\no_0^{(0)} = 1$, $\no_n^{(0)} = 0$ for $n>0$, and $\no_n^{(1)} = \no_n$).

\begin{thm}\label{thm:asymp_so^(m)}
 For any positive integer $r$, the probability that decomposition~\eqref{eqn:full_decomposition} of a uniform random permutation $\sigma \in S_n$ consists of $(2m+1)$ summands, as $n\to\infty$, satisfies
 \begin{equation}\label{eqn:asymp_so^(m)}
  \mathbb{P}(\sigma\mbox{ consists of }2m+1\mbox{ blocks})
   = 
  \sum\limits_{k=1}^{r-1}
  \dfrac{\no_k^{(m)} - \no_k^{(m+1)}}{n^{\underline{2k}}}
   +
  O\left(\dfrac{1}{n^{2r}}\right),
 \end{equation}
 where $n^{\underline{k}} = n(n-1)\ldots(n-k+1)$ are the falling factorials.
\end{thm}
\begin{proof}
 Due to Lemma~\ref{lem:SO-decomposition} and Theorem~\ref{thm:full_decomposition}, the number of permutations whose decomposition~\eqref{eqn:full_decomposition} consists of at least $(2m+1)$ summands is equal to
 \[
  \sum\limits_{k=1}^{\lfloor n/2\rfloor} \no_k^{(m)} \cdot (n-2k)!.
 \]
 Therefore, the sequence $\big(\so_n^{(m)}\big)$ satisfies
 \[
  \so_n^{(m)}
   =
  \sum\limits_{k=1}^{\lfloor n/2\rfloor}
   \left( \no_k^{(m)} - \no_k^{(m+1)} \right) \cdot (n-2k)!.
 \]
 Similarly to the previous theorem, for any tail of this sum, we have an estimation
 \[
  \sum\limits_{k=r}^{\lfloor n/2\rfloor}
   \left( \no_k^{(m)} - \no_k^{(m+1)} \right) \cdot (n-2k)!
   =
  O\big((n-2r)!\big).
 \] 
 Thus, dividing the above relation by $n!$, we obtain asymptotic expression~\eqref{eqn:asymp_so^(m)}.
\end{proof}

\section{Asymptotics of pattern distributions}\label{sec:pattern_asymptotics}

Let $\pi$ be a permutation of size $p$.
Denote by $a_{n,m}(\pi)$ the number of permutations of size $n$ with exactly $m$ very tight occurrences of the pattern $\pi$.
Assume that both $\pi$ and its reverse $\overline{\pi}$ are non-self-overlapping.
In this case, as it was showed by Myers~\cite{Myers2002}, we have
\begin{equation}\label{eq:Myers}
 a_{n,m}(\pi) = 
 \sum\limits_{k=m}^{\lfloor n/(p-1)\rfloor}
  (-1)^{m-k}\binom{k}{m}
  \binom{n-(p-1)k}{k}\big(n-(p-1)k\big)!.
\end{equation}
The goal of this section is to answer the following Myers' question: for a fixed constant $m$,
what is the asymptotics of the probability that a pattern $\pi$
very tightly occurs exactly $m$ times in a permutation of size $n$, as $n\to\infty$?
A possible answer is given by the following theorem.

\begin{thm}\label{thm:patt_asymp_1}
 If $\pi \in S_p$ is such a permutation that both $\pi$ and $\overline{\pi}$ are non-self-overlapping,
 then for any positive integer $r$, the probability that a uniform random permutation $\sigma \in S_n$ has exactly $m$ very tight occurrences of $\pi$, as $n\to\infty$, satisfies
 \begin{equation}\label{eqn:patt_asymp_1}
  \mathbb{P}(\pi\mbox{ very tightly occurs in }\sigma\mbox{ }m\mbox{ times})
   = \dfrac{1}{m!}
  \sum\limits_{k=m}^{r-1}
  \dfrac{(-1)^{k-m}}{(k-m)!} \cdot
  \dfrac{\big(n-(p-1)k\big)^{\underline{k}}}{n^{\underline{(p-1)k}}}
   +
  O\left(\dfrac{1}{n^{(p-2)r}}\right),
 \end{equation}
 where $n^{\underline{k}} = n(n-1)\ldots(n-k+1)$ are the falling factorials.
 In particular,
 \[
  \mathbb{P}(\pi\mbox{ very tightly occurs in }\sigma\mbox{ }m\mbox{ times})
   =
  \dfrac{1}{m!} \cdot \dfrac{1}{n^{(p-2)m}}
   +
  O\left(\dfrac{1}{n^{(p-2)m+1}}\right).
 \]
\end{thm}
\begin{proof}
 Relation~\eqref{eq:Myers} implies that
 \[
  a_{n,m}(\pi) = \dfrac{1}{m!}
  \left(
  \sum\limits_{k=m}^{r-1}
   \dfrac{(-1)^{k-m}}{(k-m)!} \cdot d_k(n)
   +
  \sum\limits_{k=r}^{\lfloor n/(p-1)\rfloor}
   \dfrac{(-1)^{k-m}}{(k-m)!} \cdot d_k(n)
  \right),
 \]
 where
 \(
  d_k(n) = \big(n-(p-1)k\big)^{\underline{k}} \cdot \big(n-(p-1)k\big)!
 \).
 Clearly, $d_k(n) \leqslant d_{r+1}(n)$ for any $k > r+1$. Therefore,
 \[
  \sum\limits_{k=r}^{\lfloor n/(p-1)\rfloor}
   \dfrac{1}{(k-m)!} \cdot d_k(n)
   \leqslant
  \dfrac{d_r(n)}{(r-m)!} + \dfrac{d_{r+1}(n)}{(r-m)!} \cdot
   \left(\left \lfloor\dfrac{n}{p-1}\right\rfloor - r\right)
   =
  O\big(d_r(n)\big).
 \]
 Thus, we obtain
 \[
  a_{n,m}(\pi) = \dfrac{1}{m!}
  \sum\limits_{k=m}^{r-1}
   \dfrac{(-1)^{k-m}}{(k-m)!} \cdot d_k(n)
   +
  O\big(d_r(n)\big),
 \]
 which is sufficient to be divided by $n!$ to get relation~\eqref{eqn:patt_asymp_1}.
\end{proof}

In some sense, it would be more natural to have a complete asymptotic expansion of $a_{n,m}(\pi)/n!$ over the basis $1/n^{\underline{k}}$, as it is the case for self-overlapping permutations (see relations \eqref{eqn:asymp_so} and \eqref{eqn:asymp_so^(m)}).
This is indeed possible, which is reflected by the following theorem.

\begin{thm}\label{thm:patt_asymp_2}
 If $\pi \in S_p$ is such a permutation that both $\pi$ and $\overline{\pi}$ are non-self-overlapping,
 then for any positive integer $r$, the probability that a uniform random permutation $\sigma \in S_n$ has exactly $m$ very tight occurrences of $\pi$, as $n\to\infty$, satisfies
 \begin{equation}\label{eqn:patt_asymp_2}
  \mathbb{P}(\pi\mbox{ very tightly occurs in }\sigma\mbox{ }m\mbox{ times})
   = 
  \sum\limits_{k=(p-2)m}^{r-1}
  \dfrac{c_k}{n^{\underline{k}}}
   +
  O\left(\dfrac{1}{n^r}\right),
 \end{equation}
 where $n^{\underline{k}} = n(n-1)\ldots(n-k+1)$ are the falling factorials and
 \[
  c_k = (-1)^m
  \sum\limits_{i=\lceil k/(p-1)\rceil}^{\lfloor k/(p-2) \rfloor}
   \binom{i}{m}
   \dfrac{(-1)^{(p-1)i-k}}{\big((p-1)i-k\big)!}
   \binom{i}{k-(p-2)i}.
 \]
\end{thm}
\begin{ex}
 If $m=0$, then relation~\eqref{eqn:patt_asymp_2} corresponds to the probability that a uniform permutation $\sigma\in S_n$ very tightly avoids a pattern $\pi$.
 For instance, for $\pi=132$ this relation reads
 \[
  \mathbb{P}(\sigma\mbox{ very tightly avoids }\pi)
   =
  1 -
  \dfrac{1}{n} +
  \dfrac{3}{2n^{\underline{2}}} -
  \dfrac{13}{6n^{\underline{3}}} +
  \dfrac{61}{24n^{\underline{4}}} -
  \dfrac{441}{120n^{\underline{5}}} +
  \dfrac{3031}{720n^{\underline{6}}} -
  \dfrac{28813}{5040n^{\underline{7}}} +
  O\left(\dfrac{1}{n^8}\right).
 \]
\end{ex}

Before passing to the proof of Theorem~\ref{thm:patt_asymp_2}, we need to establish one technical lemma related to the rising and falling factorials, that is
 \[
  n^{\overline{k}}=n(n+1)\ldots(n+k-1)
  \qquad\mbox{and}\qquad
  n^{\underline{k}}=n(n-1)\ldots(n-k+1), 
 \]
respectively.

\begin{lem}\label{lem:tech}
 For any non-negative integers $n$, $k$ and $l$,
 \[
  (n-k)^{\overline{l}} =
  \sum\limits_{i=0}^{l}
   \dfrac{(-1)^i}{i!} \;
   k^{\underline{i}} \;
   l^{\underline{i}} \;
   n^{\overline{l-i}}.
 \]
\end{lem}
\begin{proof}
 Let us apply induction on $l$.
 The base case holds, since $(n-k)^{\overline{0}} = 1 = n^{\overline{0}}$.
 Suppose that the statement holds for some arbitrary $l$.
 Therefore,
 \begin{align*}
  (n-k)^{\overline{l+1}}
   &=
    \big(n-(k-1)\big)^{\overline{l+1}} - (l+1)\big(n-(k-1)\big)^{\overline{l}} \\
   &=
    n^{\overline{l+1}} -
    (l+1)\sum\limits_{s=0}^{k-1}(n-s)^{\overline{l}} \\
   &=
    n^{\overline{l+1}} -
    (l+1)\sum\limits_{s=0}^{k-1}
     \sum\limits_{i=0}^{l}
      \dfrac{(-1)^i}{i!} \;
      s^{\underline{i}} \;
      l^{\underline{i}} \;
      n^{\overline{l-i}} \\
   &=
    n^{\overline{l+1}} +
    \sum\limits_{i=0}^{l}
     \dfrac{(-1)^{i+1}}{i!} \;
     (l+1)^{\underline{i+1}} \;
     n^{\overline{(l+1)-(i+1)}}.
     \sum\limits_{s=0}^{k-1}
      s^{\underline{i}}.
 \end{align*}
 To finish the proof, it is sufficient to note that
 \[
  \sum\limits_{s=0}^{k-1}
   s^{\underline{i}} =
  \sum\limits_{s=0}^{k-1}
   \dfrac{(s+1)^{\underline{i+1}}-s^{\underline{i+1}}}{i+1} =
   \dfrac{k^{\underline{i+1}}}{i+1}.
 \]
\end{proof}

\begin{proof}[Proof of Theorem~\ref{thm:patt_asymp_2}]
 The main idea is to take expression~\eqref{eqn:patt_asymp_1} and rewrite it with the help of Lemma~\ref{lem:tech}.
 For simplicity, let us denote $N = n - (p-1)k + 1$.
 Thus, Lemma~\ref{lem:tech} applied to $l=k$ gives us
 \[
  \dfrac{\big(n-(p-1)k\big)^{\underline{k}}}{n^{\underline{(p-1)k}}}
   =
    \dfrac{(N-k)^{\overline{k}}}{N^{\overline{(p-1)k}}}
   =
    \sum\limits_{i=0}^{k}
     \dfrac{(-1)^{i}\;(k^{\underline{i}})^2\; N^{\overline{k-i}}}
     {i! \cdot N^{\overline{(p-1)k}}}
   =
    \sum\limits_{i=0}^{k}
     \binom{k}{i}
     \dfrac{(-1)^{i}\; k^{\underline{i}}\;}
     {(N+k-i)^{\overline{(p-2)k+i}}}.
 \]
 Taking into account that $(N+k-i)^{\overline{(p-2)k+i}} = n^{\underline{(p-2)k+i}}$,
 we can rewrite expression~\eqref{eqn:patt_asymp_1} as
 \begin{equation}\label{eqn:patt_asymp_3}
  \mathbb{P}(\pi\mbox{ very tightly occurs in }\sigma\mbox{ }m\mbox{ times})
   =
  \sum\limits_{k=m}^{r-1} a_k
   \sum\limits_{i=0}^{k} b_{k,i}\; e_{(p-2)k+i}
   +
  O\left(\dfrac{1}{n^{(p-2)r}}\right),
 \end{equation}
 where 
 \[
  a_k = (-1)^m \binom{k}{m},
   \qquad
  b_{k,i} = \dfrac{(-1)^{k-i}}{(k-i)!}\binom{k}{i},
   \qquad
  e_k = \dfrac{1}{n^{\underline{k}}}.
 \]
 Now, reverse the summation order in the right hand side of relation~\eqref{eqn:patt_asymp_3}.
 Designating $s = (p-2)k+i$, we have
 \begin{align*}
  \sum\limits_{k=m}^{\infty} a_k
   \sum\limits_{i=0}^{\infty} b_{k,i} \; {\bf 1}_{i\leqslant k}\; e_{(p-2)k+i}
  &=
  \sum\limits_{k=m}^{\infty} a_k
   \sum\limits_{s=(p-2)k}^{\infty}
    b_{k,s-(p-2)k}\;
    {\bf 1}_{s\leqslant(p-1)k} \;
    e_s \\
  &=
  \sum\limits_{s=(p-2)m}^{\infty} e_s
   \sum\limits_{k=\lceil s/(p-1)\rceil}^{\lfloor s/(p-2) \rfloor}
    a_k \;
    b_{k,s-(p-2)k},
 \end{align*}
 where ${\bf 1}$ is the indicator function.
 Substituting $e_s$, $a_k$ and $b_{k,s-(p-2)k}$ leads to the desired form of the asymptotics.
\end{proof}

\section{Conclusion}\label{sec:conclusion}

As we have seen in Section~\ref{sec:structure}, each permutation can be represented as a direct sum of non-self-overlapping ones.
Therefore, we can consider non-self-overlapping permutations as a sort of basic irreducible blocks that can serve for constructing the whole class of permutations.
There are other examples of this kind in the literature related to permutations.
 One of them concerns so called \emph{indecomposable permutations}, that is, permutations with no proper invariant interval of the form $\{1,\ldots,k\}$ (see, for example, the paper of Cori \cite{Cori2009}).
It is known that each permutation can be represented as a sequence of indecomposable ones.
In particular, the generating function $\IP(z)$ of indecomposable permutations satisfies
 \[
  \IP(z) = 1 - \dfrac{1}{\P(z)}.
 \]

In the class of permutations, the non-self-overlapping ones form the majority.
In other words, almost all permutations are non-self-overlapping.
The same is true for indecomposable permutations.
Moreover, this analogy can be extended to complete asymptotic expansions.
It was Comtet \cite{Comtet1972} who first studied the probability that a~large permutation is indecomposable, and established its expansion over the basis $1/n^{\underline{k}}$.
Later, it turned out that the involved coefficients have combinatorial meaning.
More precisely, the following result was recently shown \cite{Nurligareev2022}:
for any positive integer $r$, the probability that a uniform random permutation $\sigma\in S_n$ is indecomposable satisfies
 \begin{equation}\label{eqn:asymp_ip}
  \mathbb{P}\big(\sigma\mbox{ is indecomposable}\big)
   =
  1 - \sum\limits_{k=1}^{r-1}
   \dfrac{2\ip_k-\ip_k^{(2)}}{n^{\underline{k}}}
   +
  O\left(\dfrac{1}{n^r}\right),
 \end{equation}
where $(\ip_k)$ is the counting sequence of indecomposable permutations, and $(\ip_k^{(2)})$ counts permutations with exactly two indecomposable parts.
From Theorem~\ref{thm:asymp_so}, we can see that the asymptotic expansion of the probability that a uniform random permutation is non-self-overlapping has the same spirit:
 \begin{equation}\label{eqn:asymp_no}
  \mathbb{P}(\sigma\mbox{ is non-self-overlapping})
   = 
  1 - \sum\limits_{k=1}^{r-1} \dfrac{\no_k}{n^{\underline{2k}}}
   +
  O\left(\dfrac{1}{n^{2r}}\right).
 \end{equation}

Another example of irreducibilities involves \emph{simple permutations} that do not map non-trivial intervals onto intervals (see \cite{AlbertAtkinsonKlazar2003}).
The generating function $\SP(z)$ of simple permutations of size at least $4$ can be found iteratively from the equation
 \[
  \dfrac{P(z) - P^2(z)}{1 + P(z)}
   =
  z + \SP\big(\P(z)\big).
 \]
The asymptotic expansion of the probability that a uniform random permutation $\sigma \in S_n$ is simple was established by Borinsky \cite{Borinsky2018}:
 \begin{equation}\label{eqn:asymp_sp}
  \mathbb{P}(\sigma\mbox{ is simple})
   = 
  \dfrac{1}{e^2}\left(
   1
    -
   \dfrac{4}{n}
    +
   \dfrac{2}{n^{\underline{2}}}
    -
   \dfrac{40}{3n^{\underline{3}}}
    -
   \dfrac{182}{3n^{\underline{4}}}
    -
   \ldots
  \right).
 \end{equation}
Again, we observe the role of the basis $1/n^{\underline{k}}$.
However, compared to the cases related to non-self-overlapping and indecomposable permutations, this expansion has two important differences.
First, the leading term is not $1$ anymore: the proportion of simple permutations tends to $e^{-2}$.
Second, the involved coefficients are not integers anymore.

These observations raise a number of questions.
First of all, it would be of interest to unite the three above discussed permutation classes under the same theory that could explain both the similarities and differences.
Relations~\eqref{eqn:asymp_so^(m)} and~\eqref{eqn:patt_asymp_2} could also fit into this theory if we added the ability to account for various statistics.
A priori, this question is quite complicated.
It would be natural to try employing the \emph{wreath product} introduced by Atkinson and Stitt~\cite{AtkinsonStitt2002} for studying restricted permutations.
Indeed, constructing the class of permutations from the simple ones admits description in terms of the wreath product~\cite{AlbertAtkinson2005}; that is true for indecomposable permutations as well.
However, it looks like the non-self-overlapping permutations do not allow such a description.

Another question is the following: is it possible to naturally modify the basis in asymptotics~\eqref{eqn:asymp_sp}, so that the coefficients become integers?
At first glance, it seems that it could be sufficient to divide the basis vectors $1/n^{\underline{k}}$ by $k!$.
This change suggests switching from the ordinary generating functions to the exponential ones.
That trick works well for establishing complete asymptotic expansions of indecomposable permutations and indecomposable perfect matchings, see \cite[Chapter 9]{Nurligareev2022}.
The potential of its applicability for simple permutations needs to be verified. 

Finally, if we get a positive answer to the previous question, what is the combinatorial interpretation of these coefficients?
Unfortunately, there is no clue in the OEIS~\cite{oeis}.
The only thing one can be sure of is that the coefficients cannot be interpreted as a~counting sequence of combinatorial class: some of them are negative.
Thus, we would rather expect linear combinations of different counting sequences, as is the case for indecomposable permutations \eqref{eqn:asymp_ip}.

In addition, there are various questions about patterns in permutations and other structures.
Thus, one may wonder if it is possible to extend Theorems~\ref{thm:patt_asymp_1} and~\ref{thm:patt_asymp_2} to any very tight pattern.
We will address this problem and propose a solution in our next paper~\cite{KN-clusters}.
We will also show that this kind of permutation patterns can be treated similarly to the special family of matching patterns that we call \emph{endhered}.
Studying endhered patterns in matchings were originally motivated by their relations to distributions of special kind of patterns in RNA secondary structures with allowed pseudoknots, modelled as fixed-point free involutions~\cite{BianeHampikianKirgizovNurligareev2023}.
However, it turns out that endhered patterns are also of independent interest, since their asymptotic behavior is somewhat similar to that of very tight permutation patterns.

A much more ambitions goal would be to study patterns and their asymptotics in graphs.
This problem is very hard in its generality.
A possible approach would be to start with so-called \emph{ordered graphs}~\cite{DucoffeFeuilloleyHabibPitois2023}.
On the one hand, they can be considered as a generalization of matchings, which potentially allows to employ ideas that are useful for treating matchings.
On the other hand, we anticipate the need for advances techniques for both enumeration and asymptotic analysis.
This may include, but is not limited to, tools such as graphic generating functions for enumeration~\cite{DovgalPanafieu2019} and coefficient generating functions for asymptotics~\cite{DovgalNurligareev2023}.

\section{Acknowledgements}
Authors were partially supported by the grant ANR-22-CE48-0002 funded by l'Agence Nationale de la Recherche and the project ANER ARTICO funded by Bourgogne-Franche-Comté region (France).
We also would like to thank the anonymous reviewer for their thorough reading and valuable comments.

\bibliographystyle{abbrv}
\bibliography{SO-perm}

\begin{thebibliography}{10}

\bibitem{AlbertAtkinson2005}
M.~H. Albert and M.~D. Atkinson.
\newblock Simple permutations and pattern restricted permutations.
\newblock {\em Discrete Math.}, 300(1-3):1--15, 2005.

\bibitem{AlbertAtkinsonKlazar2003}
M.~H. Albert, M.~D. Atkinson, and M.~Klazar.
\newblock The enumeration of simple permutations.
\newblock {\em J. Integer Seq.}, 6(4):art. 03.4.4, 18, 2003.

\bibitem{AtkinsonStitt2002}
M.~D. Atkinson and T.~Stitt.
\newblock Restricted permutations and the wreath product.
\newblock {\em Discrete Math.}, 259(1-3):19--36, 2002.

\bibitem{BianeHampikianKirgizovNurligareev2023}
C.~Biane, G.~Hampikian, S.~Kirgizov, and K.~Nurligareev.
\newblock Endhered patterns in matchings and {RNA}.
\newblock {Accepted to J. Comput. Biol.}, 2024.

\bibitem{Bona2007}
M.~B{\'o}na.
\newblock On three different notions of monotone subsequences.
\newblock In {\em {Permutation patterns. Proceedings of the 5th conference held
  at the University of St. Andrews, Scotland, June 11--15, 2007}}, pages
  89--114. Cambridge: Cambridge University Press, 2010.

\bibitem{Bona2011}
M.~B{\'o}na.
\newblock Non-overlapping permutation patterns.
\newblock {\em Pure Math. Appl. (PU.M.A.)}, 22(2):99--105, 2011.

\bibitem{Borinsky2018}
M.~Borinsky.
\newblock {Generating asymptotics for factorially divergent sequences}.
\newblock {\em {Electron. J. Comb.}}, 25(4):p4.1, 32, 2018.

\bibitem{Claesson2022}
A.~Claesson.
\newblock From {Hertzsprung}'s problem to pattern-rewriting systems.
\newblock {\em Algebr. Comb.}, 5(6):1257--1277, 2022.

\bibitem{Comtet1972}
L.~Comtet.
\newblock {Sur les coefficients de l'inverse de la s\'erie formelle \(\sum n!
  t^n\).}
\newblock {\em {C. R. Acad. Sci., Paris, S\'er. A}}, 275:569--572, 1972.

\bibitem{Cori2009}
R.~Cori.
\newblock {Indecomposable permutations, hypermaps and labeled Dyck paths.}
\newblock {\em {J. Comb. Theory, Ser. A}}, 116(8):1326--1343, 2009.

\bibitem{DovgalPanafieu2019}
S.~Dovgal and {\'E}.~de~Panafieu.
\newblock Symbolic method and directed graph enumeration.
\newblock {\em Acta Math. Univ. Comenian. (N.S.)}, 88(3):989--996, 2019.

\bibitem{DovgalNurligareev2023}
S.~Dovgal and K.~Nurligareev.
\newblock Asymptotics for graphically divergent series: dense digraphs and
  2-sat formulae, 2023.
\newblock arXiv preprint \url{https://arxiv.org/abs/2310.05282}.

\bibitem{DuaneRemmel2011}
A.~Duane and J.~Remmel.
\newblock Minimal overlapping patterns in colored permutations.
\newblock {\em Electron. J. Comb.}, 18(2):research paper p25, 38, 2011.

\bibitem{DucoffeFeuilloleyHabibPitois2023}
G.~Ducoffe, L.~Feuilloley, M.~Habib, and F.~Pitois.
\newblock Pattern detection in ordered graphs, 2023.
\newblock arXiv preprint \url{https://arxiv.org/abs/2302.11619}.

\bibitem{Elizalde2013}
S.~Elizalde.
\newblock The most and the least avoided consecutive patterns.
\newblock {\em Proc. Lond. Math. Soc.}, 106(5):957--979, 2013.

\bibitem{Elizalde2016}
S.~Elizalde.
\newblock A survey of consecutive patterns in permutations.
\newblock In {\em Recent trends in combinatorics}, pages 601--618. Cham:
  Springer, 2016.

\bibitem{ElizaldeNoy2003}
S.~Elizalde and M.~Noy.
\newblock Consecutive patterns in permutations.
\newblock {\em Adv. Appl. Math.}, 30(1-2):110--125, 2003.

\bibitem{GouldenJackson1979}
I.~P. Goulden and D.~M. Jackson.
\newblock An inversion theorem for cluster decompositions of sequences with
  distinguished subsequences.
\newblock {\em J. Lond. Math. Soc., II. Ser.}, pages 567--576, 1979.

\bibitem{GouldenJackson2004}
I.~P. Goulden and D.~M. Jackson.
\newblock {\em Combinatorial enumeration}.
\newblock Courier Corporation, 2004.

\bibitem{KN-clusters}
S.~Kirgizov and K.~Nurligareev.
\newblock Clusters of endhered patterns in permutations and matchings.
\newblock {In preparation}.

\bibitem{MarcusTardos2004}
A.~Marcus and G.~Tardos.
\newblock Excluded permutation matrices and the {Stanley–Wilf} conjecture.
\newblock {\em {J. Comb. Theory, Ser. A}}, 107(1):153--160, 2004.

\bibitem{MonteilNurligareev2024}
T.~Monteil and K.~Nurligareev.
\newblock Asymptotic probability for connectedness, 2024.
\newblock arXiv preprint \url{https://arxiv.org/abs/2401.00818}.

\bibitem{Myers2002}
A.~N. Myers.
\newblock Counting permutations by their rigid patterns.
\newblock {\em J. Comb. Theory, Ser. A}, 99(2):345--357, 2002.

\bibitem{oeis}
{N. J. A. Sloane and The OEIS Foundation Inc.}
\newblock {The On-Line Encyclopedia of Integer Sequences}, 2024.
\newblock Published electronically at \url{https://oeis.org}.

\bibitem{NoonanZeilberger1999}
J.~Noonan and D.~Zeilberger.
\newblock The {Goulden--Jackson} cluster method: extensions, applications and
  implementations.
\newblock {\em J. Difference Equ. Appl.}, 5(4-5):355--377, 1999.

\bibitem{Nurligareev2022}
K.~Nurligareev.
\newblock {\em Irreducibility of combinatorial objects: asymptotic probability
  and interpretation}.
\newblock PhD thesis, University Paris 13, 2022.
\newblock Thèse de doctorat dirigée par Thierry Monteil et Lionel Pournin,
  Informatique, Université Paris 13.

\bibitem{PanRemmel2016}
R.~Pan and J.~B. Remmel.
\newblock Asymptotics for minimal overlapping patterns for generalized {Euler}
  permutations, standard tableaux of rectangular shape, and column strict
  arrays.
\newblock {\em Discrete Math. Theor. Comput. Sci.}, 18(2):13, 2016.
\newblock Id/No 6.

\bibitem{SimionSchmidt1985}
R.~Simion and F.~W. Schmidt.
\newblock Restricted permutations.
\newblock {\em Eur. J. Comb.}, 6(4):383--406, 1985.

\end{thebibliography}

\end{document}